\documentclass[a4paper,reqno,draft]{amsart}

\usepackage[T2A]{fontenc}
\usepackage[cp1251]{inputenc}
\usepackage{amssymb,amscd,amsthm,amsaddr}

\theoremstyle{plain}
\newtheorem{lem}{Lemma}

\theoremstyle{plain}
\newtheorem{thm}{Theorem}

\theoremstyle{definition}
\newtheorem{rem}{Remark}

\theoremstyle{plain}
\newtheorem{col}{Corollary}

\makeatletter
\renewcommand{\email}[2][]{%
	\ifx\emails\@empty\relax\else{\g@addto@macro\emails{,\space}}\fi%
	\@ifnotempty{#1}{\g@addto@macro\emails{#2\space}}%
	\g@addto@macro\emails{(\textrm{#1})}%
}
\makeatother

\begin{document}
\author[A.\,Yu.\,Ulitskaya \and O.\,L.\,Vinogradov]{A.\,Yu.\,Ulitskaya \and O.\,L.\,Vinogradov}
\address{Saint Petersburg State University,\\28, Universitetskii pr., Saint Petersburg, 198504, Russia}
\email[A.\,Yu.\,Ulitskaya]{baguadadao@gmail.com}
\email[O.\,L.\,Vinogradov]{olvin@math.spbu.ru}

\title[Optimal subspaces for mean square approximation]{Optimal subspaces for mean square approximation of classes of differentiable functions with boundary conditions}
\thanks{This work is supported by the Russian Science Foundation under grant No. 18-11-00055.}

\begin{abstract}
In this paper, we specify a set of optimal subspaces for $L_2$ approximation of three classes of functions in the Sobolev space $W^{(r)}_2$, defined on a segment and subject to certain boundary conditions. All of these subspaces are generated by equidistant shifts of a single function. In particular, we indicate optimal spline spaces of all degrees $d\geqslant r-1$ with uniform knots.
\end{abstract}
	
\maketitle

\section{Introduction}
\subsection{Notation}
In what follows, $\mathbb{C}$, $\mathbb{R}$, $\mathbb{Z}$, $\mathbb{Z}_+$, $\mathbb{N}$ are the sets of complex, real, integer, nonnegative integer, and natural numbers, respectively; $[a:b]=[a,b]\cap\mathbb{Z}$. Unless otherwise follows from the context, all the functional spaces under consideration can be real or complex.
If $p\in[1,+\infty)$, $L_p$ is the space of $2\pi$-periodic, measurable, and $p$-integrable functions; $L_p[a,b]$ is the space of measurable on $[a,b]$, $p$-integrable functions. The norms in these spaces are defined by
$$
\|f\|_p=\left(\int_{-\pi}^\pi|f|^p\right)^{1/p},\qquad
\|f\|_{L_p[a,b]}=\left(\int_a^b|f|^p\right)^{1/p},
$$
respectively.
Furthermore, $W^{(r)}_2[a,b]$ is the space of functions $f$ in $L_2[a,b]$ such that $f^{(r-1)}$ is absolutely continuous and $f^{(r)}\in L_2[a,b]$; the class $W^{(r)}_2$ of periodic functions is defined similarly. 

The symbol $\langle\cdot,\cdot\rangle_\mathcal{H}$ denotes the inner product in the Hilbert space~$\mathcal{H}$;
$$
E(f,\mathfrak{N})_p=\inf\limits_{T\in\mathfrak{N}}\|f-T\|_p
$$
is best approximation of $f$ in $L_p$ by the set $\mathfrak{N}\subset L_p$.  

The Fourier coefficients of the function~$f$ and the discrete Fourier transform of the finite sequence~$\{\beta_k\}_{k=0}^{2n-1}$
are defined by the equalities
$$c_k(f)=\frac{1}{2\pi}\int\limits_{-\pi}^\pi f(t)e^{-ikt}\,dt,\qquad
\widehat{\beta_l}=\sum\limits_{k=0}^{2n-1}\beta_ke^{-\frac{ilk\pi}{n}}.$$
The notation $f(x)\sim\sum\limits_{k\in\mathbb{Z}}c_ke^{ikx}$ means that the series on the right-hand side is the Fourier series of~$f$.

For $n\in\mathbb{N}$ and $\mu\in\mathbb{Z}_+$, let $\mathbf{S}_{n,\mu}$ 
denote the $2n$-dimensional space of $2\pi$-periodic splines of degree~$\mu$ and defect~$1$ with knots at the points $\dfrac{k\pi}{n}$, $k\in\mathbb{Z}$;
$\mathcal{T}_{2n-1}$ is the $(2n-1)$-dimensional space of trigonometric polynomials of degree at most $n-1$.

The symbols $f^e$ and $f^o$ denote the even and the odd parts of $f$, respectively, i.e.
$$
f^e=\frac{f+f(-\cdot)}{2},\qquad f^o=\frac{f-f(-\cdot)}{2}.
$$

Let $X$ be a normed linear space and $A$ a subset of $X$. The \emph{Kolmogorov $n$-width\/} of $A$ in $X$ is given by
$$
d_n(A;X)=\inf\limits_{X_n}\sup\limits_{x\in A}\inf\limits_{y\in X_n}\|x-y\|_X,
$$ 
where the external lower bound is taken over all subspaces $X_n$ of the space $X$, whose dimension does not exceed $n$.

\subsection{An overview of the results}
In~\cite{floater}, Floater and Sande studied the $L_2$ approximation of three classes of functions in $W^{(r)}_2[0,1]$, defined by certain boundary conditions. With slightly different scaling (which we hereafter adhere to) these classes are given by
\begin{align*}
H_0^r&=\{u\in W^{(r)}_2[0,\pi]\colon u^{(k)}(0)=u^{(k)}(\pi)=0,\quad 0\leqslant k<r,\quad k\text{ even}\},\\
H_1^r&=\{u\in W^{(r)}_2[0,\pi]\colon u^{(k)}(0)=u^{(k)}(\pi)=0,\quad 0\leqslant k<r,\quad k\text{ odd}\},\\
H_2^r&=\left\{u\in W^{(r)}_2\left[0,\frac{\pi}{2}\right]\colon u^{(k)}(0)=u^{(l)}\left(\frac{\pi}{2}\right)=0,\quad 0\leqslant k,l<r,\quad k\text{ even, }l\text{ odd}\right\}.
\end{align*}
The authors considered the function classes
\begin{align*}
A^r_i&=\{u\in H^r_i\colon \|u^{(r)}\|_{L_2[0,\pi]}\leqslant1\},\quad i=0,1,\\
A^r_2&=\left\{u\in H^r_2\colon \|u^{(r)}\|_{L_2\left[0,\frac{\pi}{2}\right]}\leqslant1\right\}
\end{align*}
and described the $n$-widths and certain optimal subspaces for $A^r_i$. Specifically, they showed that
$$
d_n(A^r_0)=\frac{1}{(n+1)^r},\quad d_n(A^r_1)=\frac{1}{n^r},\quad d_n(A^r_2)=\frac{1}{(2n+1)^r},
$$
and the spaces
\begin{equation}
\label{1.1}
\mathrm{span}\,\{x\mapsto\sin kx\}_{k=1}^n,\quad
\mathrm{span}\,\{x\mapsto\cos kx\}_{k=0}^{n-1},\quad
\mathrm{span}\,\{x\mapsto\sin (2k-1)x\}_{k=1}^n
\end{equation}
are optimal $n$-dimensional spaces for, respectively, $A^r_0$, $A^r_1$, and $A^r_2$. The result for $A^1_1$ was proved by 
Kolmogorov~\cite{kolm}.
In addition, the authors proved that the spaces $A^r_i$ admit optimal spline subspaces, which are defined as follows.

Let $\tau=(\tau_1,\ldots,\tau_m)$ be a knot vector such that $0<\tau_1<\ldots<\tau_m<P_i$, where $P_0=P_1=\pi$ and $P_2=\pi/2$.
Denote by $S_{d,\tau,i}$ the space of splines of degree~$d$ and defect~$1$ on $[0,P_i]$ and consider its $n$-dimensional subspaces defined by
\begin{align*}
S_{d,0}&=\{s\in S_{d,\tau_0,0}\colon s^{(k)}(0)=s^{(k)}(\pi)=0,\quad 0\leqslant k\leqslant d,\quad k\text{ even}\},\\
S_{d,1}&=\{s\in S_{d,\tau_1,1}\colon s^{(k)}(0)=s^{(k)}(\pi)=0,\quad 0\leqslant k\leqslant d,\quad k\text{ odd}\},\\
S_{d,2}&=\left\{s\in S_{d,\tau_2,2}\colon s^{(k)}(0)=s^{(l)}\left(\frac{\pi}{2}\right)=0,\quad 0\leqslant k,l\leqslant d,\quad k\text{ even, }l\text{ odd}\right\},
\end{align*}
where the knot vectors $\tau_i$ for $i=0,1,2$ are given as
\begin{align*}
\tau_0&=\begin{cases}
\left\{\frac{k\pi}{n+1}\right\}_{k=1}^n,&\quad d\text{ odd},\\
\left\{\frac{k\pi}{n+1}+\frac{\pi}{2(n+1)}\right\}_{k=0}^n,&\quad d\text{ even}, 
\end{cases}\\
\tau_1&=\begin{cases}
\left\{\frac{k\pi}{n}+\frac{\pi}{2n}\right\}_{k=0}^{n-1},&\quad d\text{ odd},\\
\left\{\frac{k\pi}{n}\right\}_{k=1}^{n-1},&\quad d\text{ even}, 
\end{cases}\\
\tau_2&=\begin{cases}
\left\{\frac{k\pi}{2n+1}+\frac{\pi}{2(2n+1)}\right\}_{k=0}^{n-1},&\quad d\text{ even},\\
\left\{\frac{k\pi}{2n+1}\right\}_{k=1}^n,&\quad d\text{ odd}. 
\end{cases}
\end{align*}
It was proved in~\cite{floater} that for any $d\geqslant r-1$ the spline spaces $S_{d,i}$ are optimal $n$-dimensional spaces for the set $A^r_i$, $i=0,1,2$.

As we can see, all the spaces $S_{d,i}$ have equidistant knots, but the form of the knot vector is determined by the degree~$d$. In this paper, we show that the classes $A^r_i$ admit optimal spline subspaces with both types of knots indicated in the definition of $\tau_i$ independently of the degree. 
Of course, boundary conditions should be relaxed or some extra conditions should be added in the remaining cases to ensure the dimension to equal~$n$.

Our technique consists in the following.
In~\cite{we}, we studied the periodic case and found a wide set of optimal subspaces generated by shifts of a single function, including spline spaces. 
The conditions of optimality in the periodic case were formulated in terms of Fourier coefficients of a function mentioned.
Now we reduce the problems for functions defined on a segment to similar problems for periodic functions and, 
using the results of~\cite{we}, specify a set of optimal subspaces in the nonperiodic situation.

\section{Spaces of shifts}

For $n\in\mathbb{N}$ and $B\in L_1$, let $\mathbb{S}_{B,n}$ be the space of functions $s$ defined on $\mathbb{R}$ and representable in the form
\begin{equation}
\label{2.1}
s(x)=\sum\limits_{j=0}^{2n-1}\beta_jB\left(x-\frac{j\pi}{n}\right),
\end{equation}
and $\mathbb{S}^\times_{B,n}$ be the space of functions in $\mathbb{S}_{B,n}$ that can be represented as~\eqref{2.1} with the additional condition 
\begin{equation}
\label{2.2}
\sum\limits_{j=0}^{2n-1}(-1)^j\beta_j=0.
\end{equation}
Substituting the Fourier series expansion of $B$ into~\eqref{2.1}, we obtain
$$
s(x)\sim\sum\limits_{j=0}^{2n-1}\beta_j\sum\limits_{l\in\mathbb{Z}}c_l(B)e^{il\left(x-\frac{j\pi}{n}\right)}=\sum\limits_{l\in\mathbb{Z}}c_l(B)\widehat{\beta}_le^{ilx}\sim
\sum\limits_{l=0}^{2n-1}\widehat{\beta}_l\Phi_{B,l}(x),
$$
where
$$
\Phi_{B,l}(x)=\Phi_{B,n,l}(x)=\frac{1}{2n}\sum\limits_{j=0}^{2n-1}e^{\frac{ilj\pi}{n}}B\left(x-\frac{j\pi}{n}\right)\sim\sum\limits_{\nu\in\mathbb{Z}}c_{l+2n\nu}(B)e^{i(l+2n\nu)x}.
$$
Clearly, $\Phi_{B,l}=\Phi_{B,l+2n}$ and condition~\eqref{2.2} is equivalent to $\widehat{\beta}_n=0$. Thus, the spaces $\mathbb{S}_{B,n}$ and $\mathbb{S}^\times_{B,n}$ coincide with
linear spans of the sets $\{\Phi_{B,l}\}_{l=0}^{2n-1}$ (or, equivalently, $\{\Phi_{B,l}\}_{l=1-n}^{n}$) and $\{\Phi_{B,l}\}_{l=1-n}^{n-1}$.
For $m\in[1:n]$, denote by $\mathbb{S}^\times_{B,n,m}$ the linear span of the set $\{\Phi_{B,l}\}_{l=1-m}^{m-1}$.

Note that the functions $\Phi_{B,l}$ are orthogonal: $\langle\Phi_{B,l},\Phi_{B,j}\rangle_{L_2}=0$ for $l\ne j$ and
$$\frac{1}{2\pi}\|\Phi_{B,l}\|_2^2=D_{B,l}=D_{B,n,l}=\sum\limits_{\nu\in\mathbb{Z}}|c_{l+2n\nu}(B)|^2.$$
The linear independence of the sets $\{B\left(\cdot-\frac{j\pi}{n}\right)\}_{j=0}^{2n-1}$ and $\{B\left(\cdot-\frac{j\pi}{n}\right)\}_{j=1-n}^{n-1}$ is equivalent to the fact that the functions $\Phi_{B,l}$ are nonzero for $l\in[1-n:n]$ and $l\in[1-n:n-1]$, respectively. In this case, the systems $\{\Phi_{B,l}\}_{l=1-n}^{n}$ and $\{\Phi_{B,l}\}_{l=1-n}^{n-1}$ form orthogonal bases in the spaces $\mathbb{S}_{B,n}$ and $\mathbb{S}^\times_{B,n}$. Orthonormal bases are constituted by the functions $\frac{1}{\sqrt{2\pi D_{B,l}}}\Phi_{B,l}$.

When $B$ is the Dirichlet kernel 
$$D_{n-1}(t)=\sum\limits_{k=1-n}^{n-1}e^{ikt},$$
we have $\mathbb{S}_{B,n}=\mathbb{S}^\times_{B,n}=\mathcal{T}_{2n-1}$,
$\mathbb{S}^\times_{B,n,m}=\mathcal{T}_{2m-1}$, $\Phi_{B,n}=0$, and for $|l|<n$, the functions $\Phi_{B,l}$ are ordinary exponents.

If $B$ is the $B$-spline
$$B_{n,\mu}(t)=\sum\limits_{k\in\mathbb{Z}}\left(\frac{e^{i\frac{\pi}{n}k}-1}{i\frac{\pi}{n}k}\right)^{\mu+1}e^{ikt},\quad \mu\in\mathbb{Z}_+$$
(henceforth, the fraction is assumed to equal $1$ whenever $k=0$), we find that $\mathbb{S}_{B,n}$ is the space of splines $\mathbf{S}_{n,\mu}$. The functions 
$$
\Phi_{B_{n,\mu},l}(x)=
\left(\frac{e^{i\frac{\pi}{n}l}-1}{i\frac{\pi}{n}}\right)^{\mu+1}\sum\limits_{\nu\in\mathbb{Z}}\frac{e^{i(l+2n\nu)x}}{(l+2n\nu)^{\mu+1}},
$$
which form an orthogonal basis in this space, are called \emph{exponential splines} (by convention, we have $\Phi_{B_{n,\mu},0}(x)=1$). The linear span of the system $\{\Phi_{B_{n,\mu},l}\}_{l=1-n}^{n-1}$ is denoted by $\mathbf{S}^\times_{n,\mu}$.

In more general nonperiodic situation exponential splines were introduced by Schoenberg; the basics of the theory and historical remarks can be found in~\cite{schnb}. The orthogonality of periodic exponential splines was noted by many authors; apparently, the earliest studies on this topic were~\cite{golomb,kamada}. The spaces $\mathbf{S}^\times_{n,\mu}$ and $\mathbb{S}^\times_{B,n}$ were considered by Vinogradov~\cite{ol1,ol2}.

The following lemma describes symmetry properties of spaces of shifts in terms of Fourier coefficients.

\begin{lem}
Let $n,m\in\mathbb{N}$, $m\leqslant n$, and $B\in L_1$. Then the following statements are equivalent.
\begin{enumerate}
	\item The inclusion $s\in\mathbb{S}^\times_{B,n,m}$ implies $s(-\cdot)\in\mathbb{S}^\times_{B,n,m}$.
	\item For any $l\in[0:m-1]$ there exists $\gamma_l\in\mathbb{C}\setminus\{0\}$ such that $\gamma_0\in\{-1,1\}$ and  
	\begin{equation}
	\label{2.22}
	c_{-l-2nk}(B)=\gamma_lc_{l+2nk}(B)\quad \text{for all $k\in\mathbb{Z}$.}
	\end{equation} 
	Moreover, if $\gamma_0=-1$, then $c_0(B)=0$.  
\end{enumerate}	
\end{lem}

\begin{proof}
The first statement means that $\Phi_{B,l}(-\cdot)\in\mathbb{S}^\times_{B,n,m}$ for all $l\in[1-m:m-1]$.
Replacing $l$ with $-l$ for convenience, rewrite the inclusion as
$$\Phi_{B,-l}(-x)=\sum\limits_{j=1-m}^{m-1}\gamma_j\Phi_{B,j}(x)$$
for some $\gamma_j$.
Since $\Phi_{B,j}$ is orthogonal to $\Phi_{B,-l}(-\cdot)$ for $j\ne l$, we have
$$\Phi_{B,-l}(-x)=\gamma_{l}\Phi_{B,l}(x).$$
Equating the Fourier coefficients, we get~\eqref{2.22}.
Replacing $l$ with $-l$ and $k$ with $-k$, we also have $c_{l+2nk}(B)=\gamma_{-l}c_{-l-2nk}(B)$.
If $\gamma_l=0$ for some $l$, then $c_{l+2nk}(B)=c_{-l-2nk}(B)=0$ and the equalities are trivially satisfied with an arbitrary $\gamma_l$. 
So, we can take $\gamma_l\ne0$.

On the other hand, if \eqref{2.22} is valid for a number $l$ with $\gamma_l\ne0$, it is also valid for a number $-l$ with $\frac{1}{\gamma_l}$. 
So, it is sufficient to consider $l\in[\,0:m-1]$.

Putting $l=0$, we conclude that $c_{-2nk}(B)=\gamma_0c_{2nk}(B)$ for all~$k$. Replacing $k$ by $-k$, we also have $c_{2nk}(B)=\gamma_0c_{-2nk}(B)$.
If $c_{2nk}(B)=0$ for all $k$, we can take $\gamma_0=1$. If $c_{2nk}(B)\ne0$ for some $k$, then $c_{-2nk}(B)\ne0$ for the same $k$, and so
$\gamma_0=\pm1$. Obviously, $\gamma_0=-1$ implies $c_0(B)=0$.
\end{proof}

Note that all even functions (and, in particular, the Dirichlet kernel) satisfy the second condition of Lemma~1 with $\gamma_l=1$ for all $l$.

For the $B$-spline, we have $\gamma_l=e^{-i\frac{\pi}{n}l(\mu+1)}$, and for the shifted $B$-spline
$\widetilde{B}_{n,\mu}={B}_{n,\mu}\bigl(\cdot-\frac{\pi}{2n}\bigr)$, the identity 
$c_k(\widetilde{B}_{n,\mu})=e^{-\frac{ik\pi}{2n}}c_k(B_{n,\mu})$ yields $\gamma_l=e^{-\frac{il\pi\mu}{n}}$.

\begin{rem}
For $l\in[1:m-1]$, we have $\Phi^e_{B,-l}=\gamma_l\Phi^e_{B,l}$ and $\Phi^o_{B,-l}=-\gamma_l\Phi^o_{B,l}$. Therefore, the space $\mathbb{S}^\times_{B,n,m}$ can be represented as
\begin{equation}
\label{2.3}
\mathbb{S}^\times_{B,n,m}=\mathrm{span}\,\{\Phi_{B,0}\}\oplus\mathrm{span}\,\{\Phi^e_{B,l}\}_{l=1}^{m-1}\oplus\mathrm{span}\,\{\Phi^o_{B,l}\}_{l=1}^{m-1}.
\end{equation}
\end{rem}

\begin{rem}
If $s\in\mathbb{S}^\times_{B,n,m}$, then $s(\cdot+\pi)\in\mathbb{S}^\times_{B,n,m}$ and, since $\Phi_{B,l}(x+\pi)=(-1)^l\Phi_{B,l}(x)$, we have
$$
\Phi^e_{B,l}(\pi-x)=(-1)^l\Phi^e_{B,l}(x),\quad \Phi^o_{B,l}(\pi-x)=(-1)^{l+1}\Phi^o_{B,l}(x).
$$
\end{rem}

\section{Main results}

The following theorem was established in \cite[Theorem 1]{we}.

\begin{thm}
Let $r,n,m\in\mathbb{N}$, $m\leqslant n$, and $B\in L_2$. Then the following statements are equivalent.
\begin{enumerate}
	\item For any function $f\in W^{(r)}_2$, the following inequality holds:
	\begin{equation}
	\label{222}
	E\bigl(f,\mathbb{S}^\times_{B,n,m}\bigr)_2\leqslant\frac{1}{m^r}\|f^{(r)}\|_2.
	\end{equation} 
 	\item The Fourier coefficients of $B$ satisfy the conditions
 	\begin{align}
 	c_l(B)\ne0 &\quad\text{for all }\ |l|\in[\,0:m-1], \cr
 	c_{2n\nu}(B)=0 &\quad\text{for all }\ \nu\in\mathbb{Z}\setminus\{0\},\cr
 	\sum\limits_{k\in\mathbb{Z}}\frac{|c_{l+2nk}(B)|^2}{\frac{1}{(l+2nk)^{2r}}-\frac{1}{m^{2r}}}\geqslant0
 	&\quad\text{for all }\ |l|\in[1:m-1].\label{3.1}
 	\end{align}
\end{enumerate}
\end{thm}

We need the following simple corollary of Theorem~1.

\begin{col}
Let $r,n,m\in\mathbb{N}$, $m\leqslant n$, and $B\in L_2$. Then the following statements are equivalent.
\begin{enumerate}
	\item For any function $f\in W^{(r)}_2$ such that $c_0(f)=0$, the following inequality holds:
	$$E\bigl(f,\mathbb{S}^\times_{B,n,m}\bigr)_2\leqslant\frac{1}{m^r}\|f^{(r)}\|_2.$$ 
	\item For all $|l|\in[1:m-1]$, we have $c_l(B)\ne0$ and
	$$\sum\limits_{k\in\mathbb{Z}}\frac{|c_{l+2nk}(B)|^2}{\frac{1}{(l+2nk)^{2r}}-\frac{1}{m^{2r}}}\geqslant0.$$
\end{enumerate}
\end{col}

\begin{proof}
For $g\in  L_2$, put $g_0(x)=\sum\limits_{\nu\in\mathbb{Z}}c_{2n\nu}(g)e^{i2n\nu x}$.
Obviously, the inequality \eqref{222}
on the whole class $W^{(r)}_2$
is equivalent to the system
$$E\bigl(f_0,\mathbb{S}^\times_{B_0,n,m}\bigr)_2\leqslant\frac{1}{m^r}\|f_0^{(r)}\|_2,$$ 
$$E\bigl(f-f_0,\mathbb{S}^\times_{B-B_0,n,m}\bigr)_2\leqslant\frac{1}{m^r}\|(f-f_0)^{(r)}\|_2.$$ 

$(1)\implies(2)$.
Let $\widetilde{B}\in L_2$ be such a function that $c_0(\widetilde{B})=1$, $c_{2n\nu}(\widetilde{B})=0$ for all $\nu\in\mathbb{Z}\setminus\{0\}$, and
$c_k(\widetilde{B})=c_k(B)$ for $k\ne 2n\nu$. 
For any $f\in  W^{(r)}_2$ we have
\begin{align}
\label{333} 
E\Bigl(f_0,\mathbb{S}^\times_{\widetilde{B}_0,n,m}\Bigr)_2&\leqslant\|f_0-c_0(f)\|_2\leqslant\frac{1}{m^r}\|f_0^{(r)}\|_2,\\
\label{444} 
E\Bigl(f-f_0,\mathbb{S}^\times_{\widetilde{B}-\widetilde{B}_0,n,m}\Bigr)_2&\leqslant\frac{1}{m^r}\|(f-f_0)^{(r)}\|_2.
\end{align}
Here \eqref{333} is trivial because $\mathbb{S}^\times_{\widetilde{B}_0,n,m}$ contains constants, while \eqref{444} holds due to assumption. 
So, \eqref{222} is valid. 
By Theorem~1, the Fourier coefficients of $\widetilde{B}$ satisfy its second assertion, 
so the same holds for $c_k(B)$ when $k\ne2n\nu$.

$(2)\implies(1)$. Let $f\in  W^{(r)}_2$, $c_0(f)=0$. Define $\widetilde{B}$ as above. 
Then, by Theorem~1, \eqref{222} holds for~$\widetilde{B}$. By \eqref{333} and \eqref{444}, it also holds for~$B$.
\end{proof}

Consider the following functional classes:
\begin{align*}
\widetilde{H}_0^r&=\{u\in W^{(r)}_2\colon u\text{ is odd}\},\\
\widetilde{H}_1^r&=\{u\in W^{(r)}_2\colon u\text{ is even}\},\\
\widetilde{H}_2^r&=\left\{u\in W^{(r)}_2\colon u\text{ is odd, } u\left(\cdot+\frac{\pi}{2}\right)\text{ is even}\right\}.
\end{align*}
Evidently, every function $u\in\widetilde{H}^r_i$ belongs to $H^r_i$. Conversely, according to the boundary conditions in the definition of the classes $H_i^r$, the $2\pi$-periodization of the odd extension of $u\in H_0^r$ to the interval $[-\pi,0]$ belongs to $\widetilde{H}_0^r$. Similarly, the $2\pi$-periodization of the even extension of $u\in H_1^r$ to the interval $[-\pi,0]$ is in $\widetilde{H}_1^r$. Consecutively extending $u\in H_2^r$ to an even (with respect to $\pi/2$) function on $[0,\pi]$ and to an odd function on $[-\pi,\pi]$, after $2\pi$-periodization we get a function belonging to $\widetilde{H}_2^r$.\\

Therefore, putting
$$
\widetilde{A}^r_i=\{u\in \widetilde{H}^r_i\colon \|u^{(r)}\|_2\leqslant1\},\quad i=0,1,2,
$$
we derive
\begin{align*}
d_n(\widetilde{A}^r_0;L_2)&=d_n(A^r_0;L_2[0,\pi])=\frac{1}{(n+1)^r},\\
d_n(\widetilde{A}^r_1;L_2)&=d_n(A^r_1;L_2[0,\pi])=\frac{1}{n^r},\\
d_n(\widetilde{A}^r_2;L_2)&=d_n\left(A^r_2;L_2\left[0,\frac{\pi}{2}\right]\right)=\frac{1}{(2n+1)^r}.
\end{align*}
Thus, we can reduce problems for nonperiodic classes to those for periodic classes and apply Theorem~1.
We will formulate our results for periodic classes (denoted with tildes).

\begin{rem}
Let $\mathsf{S}$ be a closed subspace of $L_2$ such that the condition $s\in\mathsf{S}$ implies $s(-\cdot)\in\mathsf{S}$. Then an element of best approximation of any function $u\in\widetilde{H}^r_0$ in $L_2$ by the space $\mathsf{S}$ is odd. Indeed, if $\|u-s\|_2=\inf\limits_{T\in\mathsf{S}}\|f-T\|_2$, we can write
\begin{gather*}
\|u-s\|_2\leqslant\left\|u-\frac{s-s(-\cdot)}{2}\right\|_2=
\left\|\frac{u-s}{2}+\frac{u+s(-\cdot)}{2}\right\|_2=\\
=\left\|\frac{u-s}{2}+\frac{-u(-\cdot)+s(-\cdot)}{2}\right\|_2\leqslant
\frac{1}{2}\left(\|u-s\|_2+\|u(-\cdot)-s(-\cdot)\|_2\right)=\|u-s\|_2. 
\end{gather*}
This means that all the inequalities in this chain turn into equalities. In particular, we have $\|u-s\|_2=\|u-s^o\|_2$. By the uniqueness of an element of best approximation in $L_2$, we conclude that $s$ is odd.

For the same reason, an element of best approximation of any function $u\in\widetilde{H}^r_1$ by the space $\mathsf{S}$ is even. If, in addition, the space $\mathsf{S}$ is invariant under the shift by $\pi$, an element of best approximation of $u\in\widetilde{H}^r_2$ by the space $\mathsf{S}$ possesses the same symmetry properties as the function $u$ itself. 
\end{rem}

Consider the following $m$-dimensional spaces:
\begin{align*}
\widetilde{\mathcal{S}}^0_{B,n,m}&=\mathrm{span}\,\{\Phi^o_{B,l}\}_{l=1}^{m}\qquad \text{for }m+1\leqslant n,\\
\widetilde{\mathcal{S}}^1_{B,n,m}&=\mathrm{span}\,\{\Phi_{B,0}\}\oplus\mathrm{span}\,\{\Phi^e_{B,l}\}_{l=1}^{m-1}\qquad\text{for } m\leqslant n,\\
\widetilde{\mathcal{S}}^2_{B,n,m}&=\mathrm{span}\,\{\Phi^o_{B,2l-1}\}_{l=1}^m\qquad \text{for }2m+1\leqslant n.
\end{align*}

In the following three theorems we give sufficient conditions of extremality of these spaces.

\begin{thm}
Let $r,n,m\in\mathbb{N}$, $m+1\leqslant n$, and suppose that the Fourier coefficients of a function $B\in L_2$ satisfy the following conditions.
\begin{enumerate}
	\item For any $l\in[1:m]$ there exists $\gamma_l\in\mathbb{C}\setminus\{0\}$ such that for all $k\in\mathbb{Z}$ $c_{-l-2nk}(B)=\gamma_lc_{l+2nk}(B)$.
	\item For all $\nu\in\mathbb{N}$, $c_{2n\nu}(B)=c_{-2n\nu}(B)$.
	\item For all $l\in[1:m]$, we have $c_l(B)\ne0$ and
	$$\sum\limits_{k\in\mathbb{Z}}\frac{|c_{l+2nk}(B)|^2}{\frac{1}{(l+2nk)^{2r}}-\frac{1}{(m+1)^{2r}}}\geqslant0.$$
\end{enumerate}
Then for any $u\in\widetilde{H}_0^r$,
\begin{equation}
\label{3.2}
E\bigl(u,\widetilde{\mathcal{S}}^0_{B,n,m}\bigr)_2\leqslant \frac{1}{(m+1)^r}\|u^{(r)}\|_2.
\end{equation}
\end{thm}

\begin{proof}
Since every $u\in\widetilde{H}_0^r$ has zero mean and $B$ satisfies the conditions of the second proposition of Corollary~1, we can write
\begin{equation}
\label{3.3}
E\bigl(u,
\mathbb{S}^\times_{B,n,m+1}\bigr)_2\leqslant \frac{1}{(m+1)^r}\|u^{(r)}\|_2.
\end{equation}
By Lemma~1, the space $\mathbb{S}^\times_{B,n,m+1}$ contains each function $s$ with its odd and even parts. Thus, an element of best approximation of $u\in\widetilde{H}_0^r$ by the space $\mathbb{S}^\times_{B,n,m+1}$ is odd. This impies that the space $\mathbb{S}^\times_{B,n,m+1}$ on the left-hand side of~\eqref{3.3} can be substituted for its subspace consisting of odd functions. Since $\Phi_{B,0}$ is even (by condition~2), it follows from decomposition~\eqref{2.3} that the desired approximating subspace coincides with $\widetilde{\mathcal{S}}^0_{B,n,m}$.
\end{proof}

\begin{thm}
Let $r,n,m\in\mathbb{N}$, $m\leqslant n$, and suppose that the Fourier coefficients of a function $B\in L_2$ satisfy the following conditions.
\begin{enumerate}
	\item For any $l\in[1:m-1]$ there exists $\gamma_l\in\mathbb{C}\setminus\{0\}$ such that for all $k\in\mathbb{Z}$ $c_{-l-2nk}(B)=\gamma_lc_{l+2nk}(B)$.
	\item For all $l\in[0:m-1]$, $c_l(B)\ne0$.
	\item For all $\nu\in\mathbb{Z}\setminus\{0\}$, $c_{2n\nu}(B)=0$.
	\item For all $l\in[1:m-1]$, we have
	$$\sum\limits_{k\in\mathbb{Z}}\frac{|c_{l+2nk}(B)|^2}{\frac{1}{(l+2nk)^{2r}}-\frac{1}{m^{2r}}}\geqslant0.$$
\end{enumerate}
Then for any $u\in\widetilde{H}_1^r$,
\begin{equation}
\label{3.4}
	E\bigl(u,\widetilde{\mathcal{S}}^1_{B,n,m}\bigr)_2\leqslant \frac{1}{m^r}\|u^{(r)}\|_2.
\end{equation}
\end{thm}

\begin{proof}
Applying Theorem~1 to $u\in\widetilde{H}_1^r$, we can write
$$
E\bigl(u,\mathbb{S}^\times_{B,n,m}\bigr)_2\leqslant \frac{1}{m^r}\|u^{(r)}\|_2.
$$
Using the same argument as in the proof of Theorem~2, we conclude that the space $\mathbb{S}^\times_{B,n,m}$ on the left-hand side of the last inequality can be substituted for the subspace of even functions, i.e. for $\widetilde{\mathcal{S}}^1_{B,n,m}$. 	
\end{proof}

\begin{thm}
Let $r,n,m\in\mathbb{N}$, $2m+1\leqslant n$, and suppose that the Fourier coefficients of a function $B\in L_2$ satisfy the following conditions.
\begin{enumerate}
	\item For any $l\in[1:2m]$ there exists $\gamma_l\in\mathbb{C}\setminus\{0\}$ such that for all $k\in\mathbb{Z}$ $c_{-l-2nk}(B)=\gamma_lc_{l+2nk}(B)$.
	\item For all $\nu\in\mathbb{N}$, $c_{2n\nu}(B)=c_{-2n\nu}(B)$.
	\item For all $l\in[1:2m]$, we have $c_l(B)\ne0$ and
	$$\sum\limits_{k\in\mathbb{Z}}\frac{|c_{l+2nk}(B)|^2}{\frac{1}{(l+2nk)^{2r}}-\frac{1}{(2m+1)^{2r}}}\geqslant0.$$
\end{enumerate}
Then for any $u\in\widetilde{H}_2^r$,
\begin{equation}
\label{3.5}
	E\bigl(u,\widetilde{\mathcal{S}}^2_{B,n,m}\bigr)_2\leqslant \frac{1}{(2m+1)^r}\|u^{(r)}\|_2.
\end{equation}
\end{thm}

\begin{proof}
Because $\widetilde{H}_2^r$ is a subspace of $\widetilde{H}_0^r$, by Theorem~2, for every $u\in\widetilde{H}_2^r$, we have
$$
E\bigl(u,\widetilde{\mathcal{S}}^0_{B,n,2m}\bigr)_2\leqslant \frac{1}{(2m+1)^r}\|u^{(r)}\|_2.
$$
Since $u$ satisfies the equality $u=u(\pi-\cdot)$, we can restrict the space $\widetilde{\mathcal{S}}^0_{B,n,2m}$ to a subspace of functions possessing this property. By Remark~2, we obtain that the desired subspace coincides with $\widetilde{\mathcal{S}}^2_{B,n,m}$.
\end{proof}

\begin{rem}
It follows easily that under the assumptions of Theorem~4 the space $\mathrm{span}\,\{\Phi^e_{B,2l-1}\}_{l=1}^m$ is extremal for the set defined by interchanging the roles of $k$ and $l$ in $H^r_2$ (or, equivalently, the symmetry conditions in $\widetilde{H}_2^r$).
\end{rem}

\begin{rem}
Note that the conditions of Theorems~2--4 are invariant under the shift of $B$ by $\frac{\pi}{2n}$.
\end{rem}

\begin{rem}
Inequalities~\eqref{3.2}, \eqref{3.4} and \eqref{3.5} turn into equalities for the functions $x\mapsto\sin(m+1)x$, $x\mapsto\cos mx$ and $x\mapsto\sin(2m+1)x$, respectively.
\end{rem}

The estimates from Theorems 2--4 can be strengthened in a standard way by replacing their right-hand sides with best approximations.

\begin{col}
Let $B\in W^{(r)}_2$ under the assumptions of Theorem~2.
\begin{enumerate}
	\item If $r$ is even, then for any $u\in\widetilde{H}_0^r$,
	$$E\bigl(u,\widetilde{\mathcal{S}}^0_{B,n,m}\bigr)_2\leqslant \frac{1}{(m+1)^r}E\bigl(u^{(r)},\widetilde{\mathcal{S}}^0_{B^{(r)},n,m}\bigr)_2.$$
	\item If $r$ is odd, then for any $u\in\widetilde{H}_0^r$,
	$$E\bigl(u,\widetilde{\mathcal{S}}^0_{B,n,m}\bigr)_2\leqslant \frac{1}{(m+1)^r}E\bigl(u^{(r)},\mathrm{span}\,\{\Phi^e_{B^{(r)},l}\}_{l=1}^m\bigr)_2.$$
\end{enumerate}
\end{col}
\begin{proof}
If $r$ is even, denote by $s$ an element of best approximation of the function $u^{(r)}$ by the space $\widetilde{\mathcal{S}}^0_{B^{(r)},n,m}$. For odd $r$, let $s$ be an element of best approximation of $u^{(r)}$ by the space $\mathrm{span}\,\{\Phi^e_{B^{(r)},l}\}_{l=1}^m$. Since $c_0(s)=0$, the function $s$ has $2\pi$-periodic $r$th primitive, which we denote by $s_r$. For any $l\in[1:m]$,
$$
\left(\Phi_{B,l}^o\right)^{(r)}=\begin{cases}
\Phi_{B^{(r)},l}^o,&\quad r\text{ is even},\cr
\Phi_{B^{(r)},l}^e,&\quad r\text{ is odd}
\end{cases} 
$$ 
and hence $s_r\in\widetilde{\mathcal{S}}^0_{B,n,m}$. Applying Theorem~2 to the function $u-s_r$, we obtain
\begin{gather*}
E\bigl(u,\widetilde{\mathcal{S}}^0_{B,n,m}\bigr)_2=E\bigl(u-s_r,\widetilde{\mathcal{S}}^0_{B,n,m}\bigr)_2\leqslant\frac{1}{(m+1)^r}\|u^{(r)}-s\|_2=\\
=\begin{cases}
\frac{1}{(m+1)^r}E\bigl(u^{(r)},\widetilde{\mathcal{S}}^0_{B^{(r)},n,m}\bigr)_2,&\quad r\text{ is even},\cr
\frac{1}{(m+1)^r}E\bigl(u^{(r)},\mathrm{span}\,\{\Phi^e_{B^{(r)},l}\}_{l=1}^m\bigr)_2,&\quad r\text{ is odd.}
\end{cases} 
\end{gather*}
\end{proof}

The proof of the two following statements goes exactly the same way as in Corollary~2 and therefore is omitted.

\begin{col}
Let $B\in W^{(r)}_2$ under the assumptions of Theorem~3.
\begin{enumerate}
	\item If $r$ is even, then for any $u\in\widetilde{H}_1^r$,
	$$E\bigl(u,\widetilde{\mathcal{S}}^1_{B,n,m}\bigr)_2\leqslant \frac{1}{m^r}E\bigl(u^{(r)},\mathrm{span}\,\{\Phi^e_{B^{(r)},l}\}_{l=1}^{m-1}\bigr)_2.$$
	\item If $r$ is odd, then for any $u\in\widetilde{H}_1^r$,
	$$E\bigl(u,\widetilde{\mathcal{S}}^1_{B,n,m}\bigr)_2\leqslant \frac{1}{m^r}E\bigl(u^{(r)},\widetilde{\mathcal{S}}^0_{B^{(r)},n,m-1})_2.$$
\end{enumerate}
\end{col}

\begin{col}
Let $B\in W^{(r)}_2$ under the assumptions of Theorem~4.
\begin{enumerate}
	\item If $r$ is even, then for any $u\in\widetilde{H}_2^r$,
	$$E\bigl(u,\widetilde{\mathcal{S}}^2_{B,n,m}\bigr)_2\leqslant \frac{1}{(2m+1)^r}E\bigl(u^{(r)},\widetilde{\mathcal{S}}^2_{B^{(r)},n,m}\bigr)_2.$$
	\item If $r$ is odd, then for any $u\in\widetilde{H}_2^r$,
	$$E\bigl(u,\widetilde{\mathcal{S}}^2_{B,n,m}\bigr)_2\leqslant \frac{1}{(2m+1)^r}E\bigl(u^{(r)},\mathrm{span}\,\{\Phi^e_{B^{(r)},2l-1}\}_{l=1}^m\bigr)_2.$$
\end{enumerate}
\end{col}

\section{Examples}

In~\cite[Theorem 2]{we}, we specified an easily verifiable condition that is sufficient for the fulfilment of inequality~\eqref{3.1} (and, therefore, of the corresponding conditions of Theorems~2--4). Namely, inequality~\eqref{3.1} holds for all functions $B$ possessing the property
$$
|l+2nk|^r|c_{l+2nk}(B)|\leqslant|l|^r|c_l(B)|\quad\text{for all } |l|\in[1:m-1],\, k\in\mathbb{Z}.
$$
Among examples of functions $B$ satisfying this condition for all $m\leqslant n$ are the functions with coefficients of the form
$$
c_k(B)=\left(\frac{e^{i\frac{\pi}{n}k}-1}{i\frac{\pi}{n}k}\right)^{\mu+1}\!\eta_k, \quad
\mu\in\mathbb{Z}_+,\,\mu+1\geqslant r,
$$
where $|\eta_{l+2nk}|\leqslant|\eta_{l}|$ and $\eta_l\ne0$ for $|l|<n$. For $\eta_k=1$ we get the $B$-spline.
If $\eta_k$ are the Fourier coefficients of the function $K\in L_1$, the function $B$ is the Steklov average of order $\mu+1$ of $K$. For example, $K$ can be the Poisson kernel ($\eta_k=e^{-\alpha|k|}$, $\alpha>0$), the heat kernel ($\eta_k=e^{-\alpha k^2}$, $\alpha>0$), the kernels of some differential operators ($\eta_k=\frac{1}{P(ik)}$, where $P$ is a polynomial with only real roots), and the generalized Bernoulli kernel ($\eta_k=|k|^{-s}e^{-i\beta\mathop{\mathrm{sign}}{k}}$, $s>0$, $\beta\in\mathbb{R}$); in the latter two examples, $\eta_0$ is assumed to equal $1$.

Taking the Dirichlet kernels of appropriate order as a function $B$ in Theorems 2--4,
we get the optimal subspaces of trigonometric polynomials~\eqref{1.1}.

Now we describe spline spaces arising from Theorems 2--4 and show that the results of~\cite{floater}
follow from these theorems.

Recall that for a given space of periodic functions with appropriate symmetry conditions, we denote the space of their restrictions to 
$[\,0,\pi]$ or to $[\,0,\pi/2]$ by the same letter but without tilde.

1. Replace $n$ with $n+1$ in Theorem~2 and take $m=n$, $B=B_{n+1,d}$.
Then our space  $\widetilde{\mathcal{S}}^{0}_{B,n+1,n}$ is the $n$-dimensional space of odd splines from $\mathbb{S}^{\times}_{B,n+1}$.
Consider the space $Q_{d,1}$ of splines $s$ which have knots $\left\{\frac{k\pi}{n+1}\right\}_{k=1}^{n}$
and satisfy the boundary conditions
$$s^{(k)}(0)=s^{(k)}(\pi)=0,\quad 0\leqslant k< d,\quad k\text{ even}.$$
Its dimension equals $n$ for $d$ odd and equals $n+1$ for $d$ even.
So, for $d$ odd we have $\mathcal{S}^{0}_{B,n+1,n}=Q_{d,1}=S_{d,0}$.
For $d$ even, $\mathcal{S}^{0}_{B,n+1,n}$ is an $n$-dimensional subspace of~$Q_{d,1}$.

2. Replace $n$ with $n+1$ in Theorem~2 and take $m=n$, $B=B_{n+1,d}\bigl(\cdot-\frac{\pi}{2(n+1)}\bigr)$.
Then our space  $\widetilde{\mathcal{S}}^{0}_{B,n+1,n}$ is the $n$-dimensional space of odd splines from
$\mathbb{S}^{\times}_{B,n+1}$.
Consider the space $Q_{d,2}$ of splines $s$ which have
knots $\left\{\frac{k\pi}{n+1}+\frac{\pi}{2(n+1)}\right\}_{k=0}^{n}$ and satisfy the boundary conditions
$$s^{(k)}(0)=s^{(k)}(\pi)=0,\quad 0\leqslant k\leqslant d,\quad k\text{ even}.$$
Its dimension equals $n$ for $d$ even and equals $n+1$ for $d$ odd.
Note that $0$ and $\pi$ are not the knots.
So, for $d$ even we have $\mathcal{S}^{0}_{B,n+1,n}=Q_{d,2}=S_{d,0}$.
For $d$ odd, $\mathcal{S}^{0}_{B,n+1,n}$ is an $n$-dimensional subspace of~$Q_{d,2}$.

3. Take $m=n$, $B=B_{n,d}$ in Theorem~3.
Then our space  $\widetilde{\mathcal{S}}^{1}_{B,n,n}$ is the $n$-dimensional space of even splines from $\mathbb{S}^{\times}_{B,n}$.
Consider the space $Q_{d,3}$ of splines $s$ which have knots $\left\{\frac{k\pi}{n}\right\}_{k=1}^{n-1}$
and satisfy the boundary conditions
$$s^{(k)}(0)=s^{(k)}(\pi)=0,\quad 0\leqslant k< d,\quad k\text{ odd}.$$
Its dimension equals $n$ for $d$ even and equals $n+1$ for $d$ odd.
So, for $d$ even we have $\mathcal{S}^{1}_{B,n,n}=Q_{d,3}=S_{d,1}$.
For $d$ odd, $\mathcal{S}^{1}_{B,n,n}$ is an $n$-dimensional subspace of~$Q_{d,3}$.

4. Take $m=n$, $B=B_{n,d}\bigl(\cdot-\frac{\pi}{2n}\bigr)$ in Theorem~3.
Then our space $\widetilde{\mathcal{S}}^{1}_{B,n,n}$ is the $n$-dimensional space of even splines from
$\mathbb{S}^{\times}_{B,n}$.
Consider the space $Q_{d,4}$ of splines $s$ which have
knots $\left\{\frac{k\pi}{n}+\frac{\pi}{2n}\right\}_{k=0}^{n-1}$ and satisfy the boundary conditions
$$s^{(k)}(0)=s^{(k)}(\pi)=0,\quad 0\leqslant k\leqslant d,\quad k\text{ odd}.$$
Its dimension equals $n$ for $d$ odd and equals $n+1$ for $d$ even.
Note that $0$ and $\pi$ are not the knots.
So, for $d$ odd we have $\mathcal{S}^{1}_{B,n,n}=Q_{d,4}=S_{d,1}$.
For $d$ even, $\mathcal{S}^{1}_{B,n,n}$ is an $n$-dimensional subspace of~$Q_{d,4}$.

5. Replace $n$ with $2n+1$ in Theorem~4 and take $m=n$, $B=B_{2n+1,d}$.
Consider the space $Q_{d,5}$ of splines $s$ which have knots $\left\{\frac{k\pi}{2n+1}\right\}_{k=1}^{n}$
and satisfy the boundary conditions
$$s^{(k)}(0)=s^{(l)}\left(\frac{\pi}{2}\right)=0,\quad 0\leqslant k< d,\ 0\leqslant l\leqslant d,\quad k\text{ even, }l\text{ odd}.$$
Its dimension equals $n$ for $d$ odd and equals $n+1$ for $d$ even.
So, for $d$ odd we have $\mathcal{S}^{2}_{B,2n+1,n}=Q_{d,5}=S_{d,2}$.
For $d$ even, $\mathcal{S}^{2}_{B,2n+1,n}$ is an $n$-dimensional subspace of~$Q_{d,5}$.

6. Replace $n$ with $2n+1$ in Theorem~4 and take $m=n$, $B=B_{2n+1,d}\bigl(\cdot-\frac{\pi}{2(2n+1)}\bigr)$.
Consider the space $Q_{d,6}$ of splines $s$ which have
knots $\left\{\frac{k\pi}{2n+1}+\frac{\pi}{2(2n+1)}\right\}_{k=0}^{n-1}$ and satisfy the boundary conditions
$$s^{(k)}(0)=s^{(l)}\left(\frac{\pi}{2}\right)=0,\quad 0\leqslant k\leqslant d,\ 0\leqslant l<d,\quad k\text{ even, }l\text{ odd}.$$
Its dimension equals $n$ for $d$ even and equals $n+1$ for $d$ odd.
So, for $d$ even we have $\mathcal{S}^{2}_{B,2n+1,n}=Q_{d,6}=S_{d,2}$.
For $d$ odd, $\mathcal{S}^{2}_{B,2n+1,n}$ is an $n$-dimensional subspace of~$Q_{d,6}$.

Taking other values of $m$ and $n$ and taking the $B$-spline (shifted or not) as a function $B$ in Theorems 2--4
we get new families of optimal spline subspaces with equidistant knots.

\end{document}